\documentclass[11pt]{amsart}
\usepackage{amsmath}
\usepackage{amssymb}
\usepackage{amscd}

\def\NZQ{\mathbb}               
\def\NN{{\NZQ N}}
\def\QQ{{\NZQ Q}}
\def\ZZ{{\NZQ Z}}
\def\RR{{\NZQ R}}

\def\P{\mathcal P}

\def\gr{\mbox{\rm gr}}

\newtheorem{Theorem}{Theorem}[section]
\newtheorem{Lemma}[Theorem]{Lemma}
\newtheorem{Corollary}[Theorem]{Corollary}
\newtheorem{Proposition}[Theorem]{Proposition}

\newtheorem{Example}[Theorem]{Example}

\let\epsilon\varepsilon
\let\phi=\varphi
\let\kappa=\varkappa

\def \a {\alpha}

\begin{document}

\title{Growth of Rank 1 Valuation Semigroups}

\author{Steven Dale Cutkosky, Kia Dalili and Olga Kashcheyeva}
\thanks{The first author was partially supported by NSF }

\address{Steven Dale Cutkosky, Department of Mathematics,
328 Mathematical Sciences Bldg,
University of Missouri,
Columbia, MO 65211 USA
}\email{cutkoskys@missouri.edu}

\address{Kia Dalili, Department of Mathematics,
103 Mathematical Sciences Bldg,
University of Missouri,
Columbia, MO 65211 USA
}\email{kia@math.missouri.edu}

\address{Olga Kashcheyeva, Department of Mathematics,
University of Illinois, Chicago,
Chicago, Il 60607 USA}
\email{olga@math.uic.edu}

\maketitle

Let $(R,m_R)$ be a local domain, with quotient field $K$. Suppose that $\nu$ is a valuation of $K$
with valuation ring $(V,m_V)$, and that $\nu$ dominates $R$; that is, $R\subset V$ and $m_V\cap
R=m_R$. The possible value groups $\Gamma$ of $\nu$ have been extensively studied and classified,
including in the papers MacLane \cite{M}, MacLane and Schilling \cite{MS}, Zariski and Samuel
\cite{ZS}, and Kuhlmann \cite{K}. $\Gamma$ can be any ordered abelian group of finite rational rank (Theorem 1.1 \cite{K}). The
semigroup
$$
S^R(\nu)=\{\nu(f)\mid f\in m_R\backslash \{0\}\}
$$
is however not well understood, although it is known to encode important information about the
topology and  resolution of singularities of $\mbox{Spec}(R)$ and the ideal theory of $R$.

In Zariski and Samuel's classic book on Commutative Algebra \cite{ZS}, two general facts about the
semigroup $S^R(\nu)$ are proven (in Appendix 3 to Volume II).
\begin{enumerate}
\item[1.] $S^R(\nu)$ is a well ordered subset of the positive part of the value group $\Gamma$ of
$\nu$ of ordinal type at most $\omega^h$, where $\omega$ is the ordinal type of the well ordered
set $\NN$, and $h$ is the rank of $\nu$. \item[2.] The rational rank of $\nu$ plus the
transcendence degree of $V/m_V$ over $R/m_R$ is less than or equal to the dimension of $R$.
\end{enumerate}
The second condition is the Abhyankar inequality \cite{Ab}.

The only semigroups which are realized by a valuation on a one dimensional regular local ring are
isomorphic to the natural numbers. The semigroups which are realized by a valuation on a regular
local ring of dimension 2 with algebraically closed residue field are much more complicated, but
are completely classified by Spivakovsky in \cite{S}. A different proof is given by Favre and
Jonsson in \cite{FJ}, and the theorem is formulated in the context of semigroups by Cutkosky and
Teissier \cite{CT}.

In \cite{CT}, Teissier and the first author give some examples showing that some surprising semigroups of
rank $>1$ can occur as   semigroups of valuations on noetherian domains, and raise the general
questions of finding new constraints on value semigroups and classifying semigroups which occur as
value semigroups. 

In this paper, we consider semigroups of rank 1 valuations. We show in Theorem
\ref{Theorem1} that the Hilbert polynomial of $R$ gives a bound on the growth
of the valuation semigroup $S^R(\nu)$. This allows us to give (in Corollary
\ref{Corollary3}) a very simple example of a well ordered subsemigroup of $\QQ_+$ of ordinal type
$\omega$, which is not a value semigroup of a local domain. This shows that the above conditions 1
and 2 do not characterize value semigroups on  local domains.

The simple bound of Theorem \ref{Theorem1} of this paper is extended 
in the article \cite{CT2} of Teissier and the first author to give
a very general bound on the growth of a value semigroup of arbitrary rank,
from which a rough description of the (extremely bizzare) shape of a higher rank
valuation semigroup is derived.

Prior to this paper and \cite{CT2}, no other general constraints were known on the value  semigroups
$S^R(\nu)$. In fact, it was even unknown if the above conditions 1 and 2 characterize value
semigroups.

With our restriction that $\nu$ has rank 1, we can assume that $S^R(\nu)$ is embedded in $\RR_+$. We can further assume that $s_0=1$ is the smallest element of $S^R(\nu)$. For $n\in\NN$, let
$$
\phi(n)=|S^R(\nu)\cap(0,n)|.
$$ Corollary \ref{Corollary1} of Theorem \ref{Theorem1} shows that for $n\gg0$, $$
\phi(n)<P_R(n),
$$
 where $P_R(n)$ is the Hilbert polynomial of $R$. This bound is the best possible for a one dimensional
local domain, as we show after Corollary \ref{Corollary3}. However, this bound
is far from being sharp for $R$ of dimension larger than one.
Let
$$
\P_n=\{f\in R\mid \nu(f)\ge n\},
$$
an ideal in $R$ which contains $m_R^n$. Suppose that $R$ contains a field $k$
isomorphic to $R/m_R$,
and $R/m_R\cong V/m_V$. Then for $n\gg 0$,
$$
\phi(n)=P_R(n)-\ell(\P_n/m_R^n),
$$
where $\ell(\P_n/m_R^n)$ is the length of $\P_n/m^n$. We approximate
$\ell(\P_n/m_R^n)$  to show in Corollary \ref{CorN1} that 
$$
\mbox{lim sup }\frac{\phi(n)}{n^d}<\frac{e(R)}{d!} =\lim_{n\rightarrow\infty}\frac{P_R(n)}{n^d}
$$
whenever $R$ has dimension $d\ge 2$, where $e(R)$ is the multiplicity of $R$.
When the dimension $d$ of $R$ is greater than 1, this is significantly smaller than the upper bound given by the Hilbert polynomial
$P_R(n)$ of $R$.

In Section \ref{Growth}, we consider the rates of growth which are possible for
the function $\phi(n)$. We show that quite exotic behavior can occur, giving
 examples (Examples \ref{EXF1}, \ref{Ex1} and \ref{Exbiggrowth}) of valuations $\nu$ dominating a regular local ring $R$ of dimension two which have growth rates $n^{\alpha}$ for any
$\alpha\in\QQ$ with $1\le \alpha\le 2$. We also give an example of $n\mbox{log }n$ growth in Example \ref{Ex2}.

In the final section, we consider the general question of characterizing
rank 1 value semigroups, and ask if the necessary condition on a well ordered subsemigroup $S$ of $\RR_+$ that the growth of $|S\cap (0,n)|$ is bounded above by a polynomial in $n$ is sufficient for $S$ to be a valuation semigroup.

\section{Notation}\label{Not} 
The following conventions will hold throughout this paper.

If $G$ is a totally ordered abelian group, then $G_+$ will denote the positive elements
of $G$. $G_{\ge 0}$ will denote the non negative elements. If $a,b\in G$, we set
$$
(a,b)=\{x\in G\mid a<x<b\}.
$$

$(R,m_R)$ will denote a (Noetherian) local ring with maximal ideal $m_R$, and   $\ell(N)$ will denote the length of an $R$ module $N$.
Let
$$
P_R(n)=\frac{e(R)}{d!}n^d+\mbox{ lower order terms}
$$
be the Hilbert Samuel polynomial of $R$, where $d$ is the dimension of $R$ and $e(R)$ is the multiplicity of $R$. We have that
$\ell(R/m_R^n)=P_R(n)$ for $n\gg0$.

 Suppose that $R$ is a local domain with quotient field $K$, and 
$\nu$ is a valuation of $K$
with valuation ring $(V,m_V)$.  We will say that $\nu$ dominates $R$ if  $R\subset V$ and $m_V\cap
R=m_R$. We define the value semigroup of $\nu$ on $R$ to be
$$
S^R(\nu)=\nu(m_R-\{0\}).
$$
Let $\Gamma$ be the valuation group of $\nu$. For $\lambda\in\Gamma$,
we define ideals in $R$
$$
\P_{\lambda}=\{f\in R\mid \nu(f)\ge\lambda\}.
$$

\section{bounds for growth of Semigroups of rank 1 valuations}\label{Bound}

The bounds in this section are valid for valuations of arbitrary rank, but 
since they give information about the smallest segment of the value group,
they are essentially statements about rank 1 valuations. We use here, and in
the remainder of this paper the notation introduced above in Section \ref{Not}.

\begin{Theorem}\label{Theorem1} Suppose that $R$ is a  local domain which is dominated by a valuation
$\nu$, and suppose that $s_0$ is the smallest element of $S^R(\nu)$. Then
$$
|S^R(\nu)\cap (0,ns_0)|<\ell(R/m_R^n)
$$
for all $n\in\NN$.
\end{Theorem}
\begin{proof}
Suppose that $n\in\NN$. Since $S^R(\nu)$ is well ordered, $(0,ns_0)\cap S^R(\nu)$ is a finite set
$$
\lambda_1<\cdots<\lambda_r
$$
for some $r\in\NN$. Set $\lambda_{r+1}=ns_0$. We have a sequence of inclusions of ideals (as defined in Section \ref{Not})
\begin{equation}\label{eq3}
m_R^n\subset \P_{ns_0}=\P_{\lambda_{r+1}}\subset\P_{\lambda_r}\subset\cdots\subset \P_{\lambda_1}=m_R.
\end{equation}
Thus
\begin{equation}\label{eq20}
\sum_{i=1}^r\ell(\P_{\lambda_i}/\P_{\lambda_{i+1}})\le\ell(m_R/m_R^n).
\end{equation}
Since $\ell(\P_{\lambda_i}/\P_{\lambda_{i+1}})>0$ for all $i$, we have the desired inequality.
\end{proof}

Recall that $P_R(n)$ is the Hilbert polynomial of a local ring $R$.

\begin{Corollary}\label{Corollary1} Suppose that $R$ is a  local domain of dimension $d$ which is dominated by a valuation
$\nu$, and  $s_0$ is the
smallest element of $S^R(\nu)$. Then
\begin{enumerate}
\item[1.] For all positive integers $n\gg 0$, $|S^R(\nu)\cap (0,ns_0)|<P_R(n)$. \item[2.] There exists $c>0$ such that $|S^R(\nu)\cap
(0,ns_0)|<cn^d$ for all $n\in \NN$.
\end{enumerate}
\end{Corollary}

\begin{Corollary}\label{Corollary2}  Suppose that $R$ is a regular local ring of dimension $d$ which is dominated by a  valuation
$\nu$, and $s_0$ is
the smallest element of $S^R(\nu)$. Then
$$
|S^R(\nu)\cap (0,ns_0)|<\binom{d-1+n}{d}
$$
for all  $n\in \NN$.
\end{Corollary}

\begin{Corollary}\label{Corollary3} There exists a well ordered subsemigroup $U$ of $\QQ_+$ such that $U$ has ordinal
type $\omega$ and $U\ne S^R(\nu)$ for any valuation $\nu$ dominating a local domain $R$.
\end{Corollary}

\begin{proof} Take any subset $T$ of $\QQ_+$ such that 1 is the smallest element of $T$ and $n^n\le |T\cap
(0,n)|<\infty$ for all $n\in\NN$. For all positive integers $r$, let
$$
rT=\{a_1+\cdots+a_r\mid a_1,\ldots,a_r\in T\}.
$$
Let $U=\omega T=\cup_{r=1}^{\infty}rT$ be the semigroup generated by $T$. By our constraints,
$|U\cap(0,r)|<\infty$ for all $r\in\NN$. Thus $U$ is well ordered and has ordinal type $\omega$. By
2 of Corollary \ref{Corollary1}, $U$ cannot be the semigroup of a valuation dominating a local
domain.
\end{proof}

We will now consider more closely the bound
\begin{equation}\label{eq4}
|S^R(\nu)\cap (0,ns_0)|<P_R(n)
\end{equation}
for $n\gg 0$ of 1 of Corollary \ref{Corollary1}.

In the case when $R$ is a regular local ring of dimension 1, we have that
$$
|S^R(\nu)\cap (0,ns_0)|=n-1=P_R(n)
$$
for all $n\in\NN$, so that the bound (\ref{eq4}) is sharp.

When $R$ is an arbitrary local domain of dimension 1, the bound
(\ref{eq4}) can be far from sharp, as is shown by the following example.
Let $R$ be the localization of  $k[x,y]/y^2-x^2-x^3$ at the maximal ideal
$(x,y)$. Define a valuation $\nu$ which dominates $R$ by embedding
$R$ into the power series ring $k[[t]]$ by the $k$-algebra homomorphism
which maps $x$ to $t$ and $y$ to $t\sqrt{1+t}$. Let $\nu$ be the restriction
of the $t$-adic valuation to $R$.  Then $S^R(\nu)=\NN$ and $s_0=1$, and
$|S^R(\nu)\cap (0,ns_0)|=n-1$ for all positive $n$. 
But  $R$ has multiplicity 2, and $P_R(n)-1=2(n-1)$ for all positive $n$. 

However,  (\ref{eq4}) can be sharp for a one dimensional $R$ which is 
not regular, as is illustrated by the following example. Let $R$ be the localization of $k[x,y]/y^2-x^3$ at the maximal ideal $(x,y)$. Embed $R$ into
the valuation ring $V=k[t]_{(t)}$ by the $k$-algebra homomorphism which maps 
$x$ to $t^2$ and $y$ to $t^3$. Let $\nu$ be the restriction of  $\frac{1}{2}$ times the $t$-adic valuation on $V$ to $R$. Then 
$\nu(x)=1$, $\nu(y)=\frac{3}{2}$ and
$S^R(\nu)=\{1,\frac{3}{2},2,\frac{5}{2},3,\frac{7}{2},\ldots\}$.
Thus $s_0=1$ and 
$$
|S^R(\nu)\cap (0,ns_0)|=2(n-1)=P_R(n)-1
$$ 
for all positive $n$. 

Evidently, in the case of one dimensional domains, (\ref{eq4}) is the best bound which is always valid.

In rings of  dimension 2 and higher, (\ref{eq4}) is always far from sharp, as we show in the next section.

\section{A sharper bound}\label{Sharper}

In this section, we assume that $(R,m_R)$ is a 
local domain of dimension $d$ and multiplicity $e=e(R)$.  Suppose $\nu$ is a rank 1 valuation on the quotient field of $R$ 
with valuation ring $V$ such
that $\nu$ dominates $R$,
 and $R$ contains an infinite field $k$ such that $k\cong R/m_R$ with $R/m_R\cong V/m_V$.   Without loss of generality, we may assume that the smallest value of an
element of $m_R$ is 
$s_0 = 1$.\\  Let
$$
\phi(n)=|S^R(\nu)\cap(0,n)|
$$
for $n\in\NN$. We will measure the deviation of $\phi(n)$ from the
upper bound (\ref{eq4}) given by the
Hilbert-Samuel polynomial $P_R(n)$ of $R$.

We begin with another look at the proof of Theorem \ref{Theorem1} with these
assumptions on $R$. 
 Since  $k=V/m_V$,  we have
$$
\ell(\P_{\lambda_i}/\P_{\lambda_{i+1}})=1
$$
for all $i$ in the sequence (\ref{eq3}). For $n\in\NN$,
let 
$$
\psi(n)=\ell(\P_n/m^n).
$$
$\psi(n)$ measures the difference of $\psi(n)$ from the Hilbert-Samuel function as
$$
\phi(n)=\ell(R/m^n)-\psi(n)
$$
for all $n\in\NN$ (by (\ref{eq20})), and thus
$$
\phi(n)=P_R(n)-\psi(n)
$$
for $n\gg0$.

Let $A=\gr_m(R)$  be the associated graded ring of $R$ and
for nonzero $x\in m^i\setminus m^{i+1}$, let $\bar x$ denote the image of $x+m^{i+1}$ in $A$. We will call $\overline x$ the inital form of $x$ and $i$ the initial degree of $x$.
\\
\\

\begin{Lemma}\label{LemmaN1}
There exist $x_1,\ldots,x_d \in m \setminus m^2$ such that
$\bar{x_i}\in A$ form an algebraically independent set over $k$, and
$\nu(x_i) \neq \nu(x_j)$ for $i \neq j.$ 
\end{Lemma}

\begin{proof} Since $k$ is infinite, $A$ has a Noether
Normalization in degree one. Let $y_1, \ldots, y_d$ be elements of $m
\setminus m^2$ such that $k[\bar y_1,\ldots, \bar y_d]$ is a Noether
Normalization of $A$.

Since $\nu(a)=\nu(b)$ implies the existence of $\lambda \in k$ such
that $\nu(a+\lambda b) > \nu(a)$, we can find $\lambda_{ij} \in k$
such that $x_i= \sum \lambda_{ij}y_j$ satisfy the desired properties.  
\end{proof}

\begin{Lemma}\label{LemmaN2}
Let $x_i$'s be as in the previous lemma, and let $K$ be the fraction field of $k[x_1,\ldots,x_d]$. Then there are elements $m_1,\ldots,m_e\in R$ (where the multiplicity of $R$ is $e$) such that $\bar{m_1}, \ldots, \bar{m_e} \in A$ are linearly independent over $K$.
\end{Lemma}

\begin{proof} 

For $n\gg0$, $\ell(m^n/m^{n+1})$ is a polynomial $Q(n)$ of degree
$d-1$. We will compute the leading coefficient of $Q(n)$ in two ways.

Let $B=k[\overline x_1,\ldots,\overline x_n]$.
First, observe that $A$ is a finitely generated graded module over the standard graded ring $B$. Since $A$ has dimension $d$, and $B$ has multiplicty one,
we can compute from  tensoring a graded composition series of $A$ as a $B$  module with $K$ (or from the graded version of the additivity formula given for instance in Corollary 4.7.8 \cite{BH}) that 
the multiplicity of $A$ as a $B$ module is $\dim_K(A\otimes_{k[\bar{x_1},\ldots,\bar{x_d}]}K)$.

For $n\gg0$, $Q(n)=P_R(n+1)-P_R(n)$. Thus 
$$
Q(n)=\frac{e}{(d-1)!}n^{d-1}+\mbox{ lower order terms},
$$
and we conclude that 
$$\dim_K(A\otimes_{k[\bar{x_1},\ldots,\bar{x_d}]}K) = e.$$ 

Choose a basis for the vector space $A\otimes_{k[\bar{x_1},\ldots,\bar{x_d}]}K$
of elements of the form $\bar{m_i}\otimes1$. Such $m_i$'s have the desired property.
\end{proof}

\begin{Proposition}\label{PropN1}
Suppose the $x_i$'s and $m_i$'s are as in the previous lemmas. 
 Then the infimum limit
 \[
\mbox{\rm lim inf }\frac{\ell(\P_n/m^n)}{n^d}\ge
\frac{e}{d!}\left (1-\frac{1}{\nu(x_1)\ldots\nu(x_d)}\right).
\]
\end{Proposition}
\begin{proof} Let $\alpha_i$ be the initial degree of $m_i$.
Let 
$$
S=\{m_ix_1^{n_1}\cdots x_d^{nd}\mid \alpha_i+n_1+\cdots+n_d<n\}.
$$
We will first show that the classes of the elements of $S$ in $\P_n/m^n$
are linearly independent over $k$. 
Suppose otherwise. Then there is a nontrivial sum
\begin{equation}\label{eqN1}
\sum\lambda_{i,n_1,\ldots,n_d}m_ix_1^{n_1}\cdots x_d^{n_d}\in m^n,
\end{equation}
where $\alpha_i+n_1+\cdots+n_d<n$ and $0\ne \lambda_{i,n_1,\ldots,n_d}\in k$
for all terms in the sum.  Let $\tau$
be the smallest value of $\alpha_i+n_1+\cdots+n_d$ for a term appearing in 
(\ref{eqN1}).

Since $\tau<n$, we have that 
$$
\sum_{\alpha_i+n_1+\cdots+n_d=\tau}\lambda_{i,n_1,\ldots,n_d}m_ix_1^{n_1}\cdots x_d^{n_d}\in m^{\tau+1},
$$
and thus 
$$
\sum_{\alpha_i+n_1+\cdots+n_d=\tau}\lambda_{i,n_1,\ldots,n_d}\overline m_i\overline x_1^{n_1}\cdots \overline x_d^{n_d}=0
$$
in $m^{\tau}/m^{\tau+1}\subset A$.
But by Lemma \ref{LemmaN2}, the elements $\overline m_i\overline x_1^{n_1}\cdots \overline x_d^{n_d}$ are linearly independent over $k$ in $A$, which is a contradiction.

 Our next goal is to determine which of the elements of $S$ are in $\P_n\setminus m^n$.  Note that since the classes of these elements in
$\P_n/m^n$ are linearly independent over $k$, their number gives a lower bound on $\ell(\P_n/m^n)$.

For a fixed $i$, the condition that an element  $m_ix_1^{n_1}\ldots x_d^{n_d}$ is in $\P_n \setminus m^n$ can be written as the following system of linear inequalities in terms of $n_i$'s:
\begin{eqnarray*}
\nu(m_i)+\nu(x_1)n_1+\cdots+\nu(x_d)n_d &\geq &n \\
\alpha_i+ n_1+\cdots + n_d &<& n.
\end{eqnarray*}
Now since $\alpha_i \leq \nu(m_i)$ every solution to the following two inequalities is also a solution to the above system.
\begin{equation}\label{eqN2}
\begin{array}{lll}
\nu(x_1)n_1+\cdots+\nu(x_d)n_d &\geq &n- \alpha_i \\
n_1+\cdots + n_d &<& n-\alpha_i.
\end{array}
\end{equation}

We will now make an asymptotic approximation of the number of integral solutions
to the system (\ref{eqN2}). To a polytope $P\subset \RR^d$ and $n\in\ZZ_+$, we associate the Ehrhart function 
$$
E(p,n)=|\{z\in\ZZ^d\mid\frac{z}{n}\in P\}|.
$$
By approximating $P$ with $d$-cubes of small volume, we compute the volume of $P$
as
$$
\mbox{vol}(P)=\lim_{n\rightarrow\infty}\frac{E(P,n)}{n^d}=\lim_{n\rightarrow \infty}
\frac{|\{z\in\ZZ^d\mid z\in nP\}|}{n^d}.
$$
The volume of the $d$-simplex $\Delta$ with vertices at the origin and at distance $c_1,\ldots,c_d$ along the coordinate axes is
$$
\mbox{vol}(\Delta)=\frac{1}{d!}c_1\cdots c_d.
$$
Let $\sigma(n)$ be the set of integral solutions to the system (\ref{eqN2}).
We have that 
$$
\begin{array}{lll}
\mbox{lim inf }\frac{\ell(\P_n/m^n)}{n^d}&\ge& \lim_{n\rightarrow \infty}\frac{\sigma(n)}{n^d}\\
&=&\lim_{n\rightarrow \infty}\frac{|\{(n_1,\ldots,n_d)\in\NN^d\mid n_1+\cdots+n_d<n\}|}{n^d}\\
&&-\lim_{n\rightarrow\infty}\frac{
\{(n_1,\ldots,n_d)\in\NN^d\mid \nu(x_1)n_1+\cdots+\nu(x_d)n_d<n\}|}{n^d}\\
&=&\frac{1}{d!}(1-\frac{1}{\nu(x_1)\cdots\nu(x_d)}).
\end{array}
$$

\end{proof}

\begin{Corollary}\label{CorN1} Let assumptions be as introduced in the
beginning of this section. If the first $d$ elements in $S^R(\nu)$ are $1=s_1,s_2,\ldots, s_d,$ then the supremum limit 
\[
\mbox{\rm lim sup }\frac{\phi(n)}{n^d}<
\frac{e}{d!}\frac{1}{s_1\ldots s_d},\]
and thus, if $d>1$,
$$
\mbox{\rm lim sup }\frac{\phi(n)}{n^d}<
\frac{e}{d!}=\lim_{n\rightarrow\infty}\frac{P_R(n)}{n^d}.
$$

\end{Corollary}
\begin{proof} From the proposition it follows that there are elements $x_1, \ldots x_d \in m$ such that $\nu(x_i) \neq \nu(x_j)$ and the number of elements in  $S^R(\nu) \cap (0,n)$ is asymptotically smaller than 
\[\frac{e}{d!}\frac{1}{\nu(x_1)\ldots \nu(x_d)}n^d.\]
Since $x_i$'s have distinct values, we have 
$s_1\ldots s_d \leq \nu(x_1)\ldots \nu(x_d)$.  \end{proof}

\section{The rate of  growth of value semigroups}\label{Growth}

In this section, we study the rate of  growth of $\phi(n)=|S^R(\nu)\cap(0,n)|$
when $R$ is a regular local ring of dimension two. We show that a wide range
of interesting growth occurs within the possible ranges of $n$ and $n^2$.

We say that $\phi(n)$ has the growth rate of the function $f(n)$ if there 
exist $0<a\le b$ such that $af(n)\le\phi(n)\le bf(n)$ for all $n\gg0$.

We can easily achieve growth of $\phi(n)=|S^R(\nu)\cap(0,n)|$ of the rate   $n^d$ on a regular local ring $R$ of
dimension $d$. Choose $d$ rationally independent real positive numbers
$\gamma_1,\ldots,\gamma_d$, a regular system of parameters $x_1,\ldots, x_d$ of $R$ and prescribe
that $\nu(x_i)=\gamma_i$ for all $i$. It is also possible to achieve
growth asymptotic to  $n^d$ from a rational rank 1 valuation, as we show
in Example \ref{Exbiggrowth}. We also give examples in this section showing
that a wide range of interesting growth can be occured.

We use the following characterization of value semigroups dominating
a regular local ring of dimension two of \cite{S}, and as may also be found
with a different treatment in \cite{FJ}.  We state the characterization in the notation of \cite{CT}.

\begin{Proposition}\label{planesemigr} 
Let $S$ be a well ordered subsemigroup of $\QQ_+$ which is not isomorphic to $\NN$ and whose minimal system of generators $(1,a_1,\ldots ,a_i,\dots)$ is of ordinal type $\leq\omega$.
Let  $S_i$ denote the semigroup generated by $1,a_1,\ldots ,a_i$, and $G_i$  the subgroup of $\QQ$ which it generates. Let
$S_0=\NN_+$ and $G_0=\ZZ$. Set  $q_i=[G_i\colon G_{i-1}]$ for $i\geq 1$. Let $s_i$ be the smallest positive integer $s$ such that $sa_i\in S_{i-1}$.
Then $S$ is the semigroup $S^R(\nu)$ of a valuation $\nu$ dominating a regular local ring $R$ of dimension 2 with algebraically closed residue field if and only if
$$
 \mbox{ for each $i\geq 1$ we have  $s_i=q_i$ and $a_{i+1}>q_ia_i$.}
$$
\end{Proposition}

We need the following two statements to estimate the number of terms of a rational rank 1 semigroup $S\subset \RR$, which are contained in a fixed interval of length 1.  Our notation is $\NN=\ZZ_{\ge 0}$.

\begin{Lemma}\label{density}
Suppose that $a_1,\dots,a_k$ are positive rational numbers. Let
$G_0=\ZZ$. For a fixed $1\le i\le k$ let $S_i$ be the subsemigroup
of $\QQ_{\ge 0}$ generated by $1,a_1,\dots,a_i$ and $G_i$ be the
group $S_i+(-S_i)$. Suppose that $q_i=[G_i\colon G_{i-1}]$ and $x_i=(q_1-1)a_1+\dots+(q_i-1)a_i$.

Then for all integers $1\le i\le k$  we have $|S_i\cap(x_i-1,x_i]|\ge q_1\cdots q_i$. Moreover, for all integers $1\le i\le k$ and real numbers  $x>x_i$  we have $S_i\cap[x-1,\infty)=G_i\cap[x-1,\infty)$ and $|S_i\cap[x-1,x)|=q_1\cdots q_i$.
\end{Lemma}

\begin{proof}
Notice that $G_i=\frac{1}{q_1\cdots q_i}\ZZ$. Thus $S_i\cap[x-1,\infty)=G_i\cap[x-1,\infty)$ for all $x>x_i$ if and only if 
$|S_i\cap[x-1,x)|=q_1\cdots q_i$ for all $x>x_i$ if and only if $|S_i\cap[x-1,x)|\ge q_1\cdots q_i$ for all $x>x_i$. Also since 
$S_i+\NN=S_i$, the fact that $|S_i\cap(x_i-1,x_i]|\ge q_1\cdots q_i$ implies $|S_i\cap(x-1,x]|\ge q_1\cdots q_i$ for all $x\ge x_i$.
Moreover,  if $x$ is a fixed real number, since $S_i$ is discrete there exists $\epsilon_0>0$ such that for all $0<\epsilon\le\epsilon_0$ the following equality between sets holds $S_i\cap[x-1,x)=S_i\cap(x-\epsilon-1,x-\epsilon]$ . Therefore, the fact that 
$|S_i\cap(x_i-1,x_i]|\ge q_1\cdots q_i$ also implies $|S_i\cap[x-1,x)|\ge q_1\cdots q_i$ for all $x>x_i$.

Assume that $k=1$. Then $S_1$ can be presented as a disjoint
union of $\NN$-modules
$$S_1=\NN\cup(a_1+\NN)\cup\dots\cup((q_1-1)a_1+\NN).$$

If  $0\le j\le (q_1-1)$ then $(x_1-ja_1)\ge 0$ and $|(ja_1+\NN)\cap(x_1-1,x_1]|=|\NN\cap(x_1-ja_1-1,x_1-ja_1]|=1$. 
Thus $|S_1\cap(x_1-1,x_1]|=q_1$.

Assume that $k>1$. By induction it  suffices to assume that the
statement is true for $i\le k-1$. Notice that
$$S_k\supset S_{k-1}\cup(a_k+S_{k-1})\cup\dots\cup((q_k-1)a_k+S_{k-1}),$$
where the union on the right is a disjoint union of $S_{k-1}$-modules.

If $0\le j\le(q_k-1)$ then $(x_k-ja_k)\ge x_{k-1}$ and 
$|(ja_k+S_{k-1})\cap(x_k-1,x_k]|=|S_{k-1}\cap(x_k-ja_k-1,x_k-ja_k]|\ge q_1\cdots q_{k-1}$.
Thus, $|S_k\cap(x_k-1,x_k]|\ge q_1\cdots q_k$.
\end{proof}

\begin{Corollary}\label{estimate}
Under the assumptions of corollary \ref{density} suppose also that
$a_{i+1}>q_ia_i$ for all $1\le i\le k-1$. Then $|S_i\cap[n-1,n)|=
q_1\cdots q_i$ for all integers $n\ge q_ia_i$.
\end{Corollary}

\begin{proof}
It suffices to notice that $(q_1-1)a_1+\dots+(q_i-1)a_i<q_ia_i$ in this case.
\end{proof}

We will now give examples of value semigroups with unexpected rate of growth of the function $\phi(n)=|S\cap(0,ns_0)|$.

\begin{Example}\label{EXF1}($n\sqrt{n}$ rate of growth)

Let $R=k[x,y]_{(x,y)}$ where $k$ is an algebraically closed field. Let $\nu$ be a valuation of the quotient field of $R$ defined by its generating sequence $\{P_i\}_{i\ge 0} $ as follows

$
\begin{array}{ll}
P_0=x, & \nu(P_0)=1\\
P_1=y, & \nu(P_1)=4+\frac{1}{2}\\
P_2=P_1^2-x^{9}, & \nu(P_2)=16+\frac{1}{2^2}\\
P_3=P_2^2-x^{28}P_1, & \nu(P_3)=64+\frac{1}{2^3}\\
P_{k+1}=P_k^2-x^{7\cdot 4^{k-1}}P_{k-1},\quad & \nu(P_{k+1})=4^{k+1}+\frac{1}{2^{k+1}}.\\
\end {array}
$

Denote by $S$ the semigroup $S_R(\nu)=\nu(m_R\backslash\{0\})$. Then $\phi(n)$
grows like $n\sqrt{n}$.

\end{Example}

\begin{proof} 
We will show that $\frac{1}{6}n\sqrt{n}<\phi(n)<\frac{4}{3}n\sqrt{n}$.

Set $a_i=\nu(P_i)$ for all $i\ge 1$. Then $S$ is a subsemigroup of
$\QQ_+$ generated by $1,a_1,a_2,\dots$. With notation of Proposition
\ref{planesemigr} we have $s_i=q_i=2$ for all $i\ge 1$ and $q_ia_i\le
2\cdot 4^i+1<4^{i+1}<a_{i+1}$. This shows that $\nu$ is well
defined. Also, by corollary \ref{estimate} we find a lower bound
on $|S\cap[n-1,n)|$ for $n\ge 4^i$:
$$|S\cap[n-1,n)|\ge|S_{i-1}\cap[n-1,n)|=2^{i-1}.$$

If $n>1$ set $i=\lfloor\log_4 n\rfloor$. Then $4^i\le n<4^{i+1}$ and
$2^{i-1}\le |S\cap[n-1,n)|\le 2^i$. Thus for all $n\in\NN_+$ we
have
$$\int_{n-1}^{n}\frac{\sqrt{t}}{4}dt<\frac{\sqrt{n}}{4}<|S\cap[n-1,n)|\le\sqrt{n}<\int_{n}^{n+1}\sqrt{t}dt.$$
Then
$$|S\cap(0,n)|=|S\cap[1,n)|<\int_2^{n+1}\sqrt{t}dt=\frac{2}{3}((n+1)\sqrt{n+1}-2\sqrt 2)<\frac{4}{3}n\sqrt{n}$$
and
$$|S\cap(0,n)|=3+|S\cap[4,n)|>3+\int_4^n\frac{\sqrt{t}}{4}dt>3+\frac{1}{6}(n\sqrt{n}-8)>\frac{1}{6}n\sqrt{n}.$$

A more precise estimate can be obtained for $n=4^k$. By induction on $k\in\NN_+$ we see that  $\frac{8^k}{3}<\phi(4^k)<\frac{8^k}{2}$, since
$$\phi(4)=3,\quad\quad 8/3<3<4$$
and
$$\phi(4^{k+1})=\phi(4^k)+|S\cap[4^k,4^{k+1})|<\frac{8^k}{2}+3\cdot2^k\cdot4^k<\frac{8^{k+1}}{2}$$
and
$$\quad\phi(4^{k+1})=\phi(4^k)+|S\cap[4^k,2\cdot4^k)|+|S\cap[2\cdot4^k,4^{k+1})|>\frac{8^k}{3}+2^{k-1}\cdot4^k+2\cdot2^k\cdot4^k>\frac{8^{k+1}}{3}.$$
\end{proof}

This example can be generalized to a construction of a value semigroup $S$ such that $\phi(n)$ grows like a power function $n^\a$, where $\a\in\QQ$, with natural restriction $1<\a<2$.  

\begin{Example}\label{Ex1}($n^{\a}$ rate of growth)

Suppose that $0<p<q$ are coprime integers. Let $r=2^q$ and
$s=2^p$. Let $R=k[x,y]_{(x,y)}$ where $k$ is an algebraically closed field. Let $\nu$ be a valuation of the
quotient field of $R$ defined by its generating sequence
$\{P_i\}_{i\ge 0} $ as follows

$
\begin{array}{ll}
P_0=x, & \nu(P_0)=1\\
P_1=y, & \nu(P_1)=r+\frac{1}{s}\\
P_2=P_1^s-x^{sr+1}, & \nu(P_2)=r^2+\frac{1}{s^2}\\
P_3=P_2^s-x^{(sr-1)r}P_1, & \nu(P_3)=r^3+\frac{1}{s^3}\\
P_{k+1}=P_k^s-x^{(sr-1)r^{k-1}}P_{k-1},\quad & \nu(P_{k+1})=r^{k+1}+\frac{1}{s^{k+1}}.\\
\end {array}
$

Denote by $S$ the semigroup $S_R(\nu)=\nu(m_R\backslash\{0\})$. Then $\phi(n)$
grows like $n^{1+p/q}$.

\end{Example}

\begin{proof}
Set $a_i=\nu(P_i)$ for all $i\ge 1$. Then $S$ is a subsemigroup of
$\QQ_+$ generated by $1,a_1,a_2,\dots$. With notation of Proposition
\ref{planesemigr} we have $s_i=q_i=s$ for all $i\ge 1$ and $q_ia_i\le
s\cdot r^i+1<r^{i+1}<a_{i+1}$. This implies that $\nu$ is well defined.
Also, by corollary \ref{estimate} we find a lower bound on
$|S\cap[n-1,n)|$ for $n\ge r^i$:
$$|S\cap[n-1,n)|\ge|S_{i-1}\cap[n-1,n)|=s^{i-1}.$$

If $n>1$ set $i=\lfloor\log_r n\rfloor$. Then $r^i\le n<r^{i+1}$ and
$s^{i-1}\le |S\cap[n-1,n)|\le s^i$. Thus since $s^i=(r^i)^{p/q}$ and $s^{i-1}=\frac{(r^{i+1})^{p/q}}{s^2}$ for all $n\in\NN_+$ we
have
$$\int_{n-1}^{n}\frac{t^{p/q}}{s^2}dt<\frac{n^{p/q}}{s^2}<|S\cap[n-1,n)|\le n^{p/q}<\int_{n}^{n+1}t^{p/q}dt.$$
Then
$$|S\cap(0,n)|=|S\cap[1,n)|<\int_2^{n+1}t^{p/q}dt=\frac{q}{p+q}((n+1)^{1+p/q}-2^{1+p/q})<\frac{3q}{p+q}n^{1+p/q}<3n^{1+p/q}$$
and
$$|S\cap(0,n)|=r-1+|S\cap[r,n))|>r-1+\int_r^n\frac{t^{p/q}}{s^2}dt=r-1+\frac{q}{s^2(p+q)}(n^{1+p/q}-rs)>\frac{n^{1+p/q}}{2s^2}.$$
That is $\frac{1}{2s^2}n^{1+p/q}<\phi(n)<3n^{1+p/q}$.
\end{proof}

We remark that in the above construction it is necessary to have the strict inequality $p<q$. However, the maximal rate of growth of $n^2$ is also achievable on a rational rank 1 semigroup of a valuation centered in a 2-dimensional  polynomial ring, as we show in the next example.

\begin{Example}\label{Exbiggrowth}($n^2$ rate of growth)

Let $R=k[x,y]_{(x,y)}$ where $k$ is an algebraically closed field. Let $\nu$ be a valuation of the
quotient field of $R$ defined by its generating sequence
$\{P_i\}_{i\ge 0} $ as follows

$
\begin{array}{ll}
P_0=x, & \nu(P_0)=1\\
P_1=y, & \nu(P_1)=1+\frac{1}{2}\\
P_2=P_1^2-x^{2+1}, & \nu(P_2)=2+1+\frac{1}{2^2}\\
P_3=P_2^2-x^{2^2+2-1}P_1, & \nu(P_3)=2^2+2+\frac{1}{2}+\frac{1}{2^3}\\
P_{k+1}=P_k^2-x^{2^k+2^{k-1}-2^{k-2}}P_{k-1},\quad & \nu(P_{k+1})=2^{k}+2^{k-1}+2^{k-3}+\cdots+2^{-k-1}.\\
\end {array}
$

Denote by $S$ the semigroup $S_R(\nu)=\nu(m_R\backslash\{0\})$. Then $\phi(n)$
grows like $n^{2}$.

\end{Example}

\begin{proof} Set $a_i=\nu(P_i)$ for all $i\ge 1$. Then $S$ is the subsemigroup of $\QQ_+$ generated by $1,a_1,a_2,\ldots$. We have $q_i=2$ for all $i\ge 1$.
Solving the recursion relation, we have
$$
a_i=2^{i-1}+\frac{1}{3}(2^i-\frac{1}{2^i})
$$
for $i\ge 1$.
We have $q_{i-1}a_{i-1}=2a_{i-1}=\frac{5}{6}2^i-\frac{1}{3}\frac{1}{2^{i-2}}<2^i$.
Corollary \ref{estimate} shows that
$$
|S\cap [n-1,n)|\ge |S_{i-1}\cap [n-1,n)|=2^{i-1}
$$
for $n\ge 2^i$. For $n>1$, set $i=\lfloor \log_2n\rfloor$, so that $2^i\le n<2^{i+1}$. We have
$$
\frac{1}{4}\int_{n-1}^ntdt<\frac{n}{4}<\frac{2^{i+1}}{4}=2^{i-1}\le|S\cap[n-1,n)|.
$$
Thus for $n\ge 4$,
$$
|S\cap(0,n)|=3+|S\cap[4,n)|>3+\frac{1}{4}\int_4^ntdt>\frac{n^2}{8}.
$$
Since $|S\cap(0,n)|<\binom{1+n}{2}$ by Corollary \ref{Corollary2}, $\phi(n)$ grows at the rate of $n^2$.
\end{proof}

 Another interesting example is of logarithmic growth.

\begin{Example}\label{Ex2} ( $n\log_{10} n$ rate of growth)

Let $R=k[x,y]_{(x,y)}$ where $k$ is an algebraically closed field. Let $\nu$ be a valuation of the quotient field of $R$ defined by its generating sequence $\{P_i\}_{i\ge 0} $ as follows

$
\begin{array}{ll}
P_0=x, & \nu(P_0)=1\\
P_1=y, & \nu(P_1)=10+\frac{1}{2}\\
P_2=P_1^2-x^{21}, & \nu(P_2)=10^2+\frac{1}{2^2}\\
P_3=P_2^2-x^{190}P_1, & \nu(P_3)=10^4+\frac{1}{2^3}\\
P_{k+1}=P_k^2-x^{\a(k)}P_{k-1},\quad & \nu(P_{k+1})=10^{2^k}+\frac{1}{2^{k+1}},\quad\quad \a(k)=2\cdot 10^{2^{k-1}}-10^{2^{k-2}}.\\
\end {array}
$

Denote by $S$ the semigroup $S_R(\nu)=\nu(m_R\backslash\{0\})$. Then $\phi(n)$ grows like $n\log_{10} n$.
\end{Example}

\begin{proof}
If  $n\ge 10$ let $k=\lfloor\log_2\log_{10} n\rfloor$. Then $10^{2^k}\le n<10^{2^{k+1}}$ and  $2^k\le |S\cap[n-1,n)|\le 2^{k+1}$.  Thus for all $n\ge 10$ we have
$$\int_{n-1}^n\frac{\log_{10}t}{2}dt<\frac{\log_{10}n}{2}\le |S\cap[n-1,n)|\le 2\log_{10}n<\int_{n}^{n+1}2\log_{10}tdt$$
and
$$\frac{n\log_{10}n}{4}<9+\int_{10}^n\frac{\log_{10}t}{2}dt<|S\cap[0,n)|<9+\int_{11}^{n+1} 2\log_{10}tdt<2n\log_{10}n.$$
That is $\frac{1}{4}n\log_{10}n<\phi(n)<2n\log_{10}n$.
\end{proof}

\section{When is a semigroup a  value semigroup?}\label{When}

Corollary \ref{Corollary1} gives a necessary condition for a rank 1 well ordered semigroup $S$
consisting of positive elements of $\RR$ to be the value semigroup $S^R(\nu)$ of a valuation
dominating some local domain $R$. The condition is:
\begin{equation}\label{eq2}
\mbox{There exists $c>0$ and $d\in \ZZ_+$ such that $|S\cap (0,ns_0)|<cn^d$ for all $n$}
\end{equation}
where $s_0$ is the smallest element of $S$. An interesting question is if (\ref{eq2}) is in fact
sufficient. (\ref{eq2}) is  sufficient in the case when $d=1$, as we now show. Suppose that
$S\subset \RR_+$ is a semigroup consisting of positive elements which contains a smallest element
$s_0$. Suppose that there exists $c>0$  such that
\begin{equation}\label{eq12}
|S\cap [(0,ns_0)|<cn
\end{equation}
for all $n\in\NN$. By Lemma \ref{Lemma2} below, we may assume that $S\subset \QQ_+$ is finitely
generated by some elements $\lambda_1,\ldots,\lambda_r$. There exists $\alpha\in\QQ_+$ such that
there exists $a_i\in\NN_+$ with $\lambda_i=\alpha a_i$ for $1\le i\le r$, and
$\mbox{gcd}(a_1,\ldots,a_r)=1$. Let $k[t]$ be a polynomial ring over a field $k$. Let
$\nu(f(t))=\alpha\mbox{ord}(f(t))$ for $f(t)\in k[t]$. $\nu$ is a valuation of $k(t)$. Let $R$ be
the one dimensional local domain
$$
R=k[t^{a_1},\ldots,t^{a_r}]_{(t^{a_1},\ldots,t^{a_r})}.
$$
The quotient field of $R$ is $k(t)$, and $\nu$ dominates $R$. We have that $S=\nu(m_R\backslash \{0\})=S_R(\nu)$.

\begin{Lemma}\label{Lemma1} Suppose that $S\subset \RR_+$ is a semigroup consisting of positive elements which contains a smallest element
$s_0$. Suppose that there exists $c>0$ and $d\in\NN_+$ such that
\begin{equation}\label{eq11}
|S\cap (0,ns_0)|<cn^d
\end{equation}
for all $n\in\NN$. Then $S$ is well ordered of ordinal type $\omega$ and has rational rank $\le d$.
\end{Lemma}

\begin{proof} The fact that $S$ is well ordered of ordinal type $\omega$ is immediate from
(\ref{eq11}).

We will prove that the rational rank of $S$ is $\le d$. After rescaling $S$ by multiplying by
$\frac{1}{s_0}$, we may assume that $s_0=1$. Suppose that $t\in\NN$ and $S$ has rational rank $\ge
t$. Then there exist $\gamma_1,\ldots,\gamma_t\in S$ which are rationally independent. Let
$b\in\NN$ be such that $\mbox{max}\{\gamma_1,\ldots,\gamma_t\}<b$. For $e\in \RR_+$, we have
$$
\begin{array}{lll}
|S\cap(0,e)|&\ge&|\{a_1\gamma_1+\cdots+a_t\gamma_t\mid a_1,\ldots,a_t\in\NN\mbox{ and
}a_1\gamma_1+\cdots+a_t\gamma_t<e\}|-1\\
&\ge&|\{a_1\gamma_1+\cdots+a_t\gamma_t\mid a_i\in\NN\mbox{ and }0\le a_i<\frac{e}{tb}\mbox{ for
}1\le
i\le t\}|-1\\
&=&|\{(a_1,\ldots,a_t)\in\NN^t\mid 0\le a_i<\frac{e}{tb}\mbox{ for }1\le i<t\}|-1
\end{array}
$$
since $\gamma_1,\ldots,\gamma_t$ are rationally independent.

For $a\in\NN$ let $n=abt$. Then we see that
$$
|S\cap(0,(n+1))|\ge a^t-1=\left(\frac{1}{bt}\right)^tn^t-1.
$$
By (\ref{eq11}), we see that $t\le d$.
\end{proof}

\begin{Lemma}\label{Lemma2} Suppose that $S\subset \RR_+$ is a semigroup consisting of positive elements which contains a smallest element
$s_0$. Suppose that there exists $c>0$  such that
\begin{equation}\label{eq12}
|S\cap (0,ns_0)|<cn
\end{equation}
for all $n\in\NN$. Then $S$ is finitely generated, and the group generated by $S$ is isomorphic to
$\ZZ$.
\end{Lemma}
\begin{proof}
By Lemma \ref{Lemma1}, $S$ has rational rank 1, so we may assume that $S$ is contained in $\QQ_+$.
We may further assume that $s_0=1$. Suppose that $S$ is not finitely generated. Then for $e\in\NN$,
we can find $\lambda_1,\ldots,\lambda_e\in S$ such that $\lambda_i=\frac{a_i}{b_i}$ with
$a_i,b_i\in\NN_+$, $b_i>1$ for all $i$, $\mbox{gcd}(a_i,b_i)=1$ for all $i$ and $b_1,\ldots,b_e$
all distinct.

There exist $m_i,n_i\in\ZZ$ such that $m_ia_i+n_ib_i=1$ for $1\le i\le e$. Either $m_i>0$ and
$n_i<0$, or $m_i<0$ and $n_i>0$.  If $m_i>0$, then $|m_i|\lambda_i=|n_i|+\frac{1}{b_i}$. If
$m_i<0$, then $|m_i|\lambda_i=|n_i|-\frac{1}{b_i}$. Let $n_0=\mbox{max}\{|n_i|+1\mid 1\le i\le
i\}$. For $n\ge n_0$, we have that $n+\frac{1}{b_i}\in S\cap(n,n+1)$ if $m_i>0$, and
$(n+1)-\frac{1}{b_i}\in S\cap (n,n+1)$ if $m_i<0$. Thus $|S\cap (n,n+1)|\ge e$ for $n\ge n_0$,
which implies that
$$
|S\cap (n,n+1)|\ge en-en_0
$$
for $n\ge n_0$. For $e>c$, we have a contradiction to (\ref{eq12}).
\end{proof}

\end{document}